\documentclass[12pt]{amsart}
\usepackage{amsmath, palatino}
\usepackage{amsthm}
\usepackage{enumerate}
\usepackage{amssymb}
\usepackage[utf8]{inputenc}
\usepackage[all,cmtip]{xy}

\title[Completions and the Cuntz semigroup]{\textbf{Completions of monoids with applications to the Cuntz semigroup}}
\author{Ramon Antoine}
\author{Joan Bosa}
\author{Francesc Perera}
\address{Departament de Matem\`atiques\\ Universitat Aut\`onoma de Barcelona\\ 08193 Bellaterra\\ Barcelona, Spain.}
\email{ramon@mat.uab.cat, jbosa@mat.uab.cat, perera@mat.uab.cat} 

\date{\today}

\theoremstyle{plain}
\newtheorem{lemma}{Lemma}[section]
\newtheorem{theorem}[lemma]{Theorem}
\newtheorem{corollary}[lemma]{Corollary}
\newtheorem{proposition}[lemma]{Proposition}
\newtheorem{definition}[lemma]{Definition}
\newtheorem*{proposition*}{Proposition}
\newtheorem*{definition*}{Definition}

\newtheorem{examples}[lemma]{Examples}

\newtheorem{remark}[lemma]{Remark}

\newcommand{\W}{\mathrm{W}}
\newcommand{\V}{\mathrm{V}}
\newcommand{\T}{\mathrm{T}}
\newcommand{\N}{\mathbb{N}}
\newcommand{\Q}{\mathbb{Q}}
\newcommand{\R}{\mathbb{R} }
\newcommand{\C}{\mathbb{C}}

\newcommand{\Cu}{\textrm{Cu}}

\begin{document}

\begin{abstract}
We provide an abstract categorical framework that relates the Cuntz semigroups of the C$^*$-algebras $A$ and $A\otimes \mathcal{K}$. This is done through a certain completion of ordered monoids by adding suprema of countable ascending sequences. Our construction is rather explicit and we show it is functorial and unique up to isomorphism. This approach is used in some applications to compute the stabilized Cuntz semigroup of certain C$^*$-algebras.
\end{abstract}

\maketitle
\section*{Introduction}

Given a C$^*$-algebra $A$, the structure of the Cuntz semigroup $\W(A)$, introduced by J. Cuntz in 1978 (\cite{cu}), has been intensively studied in recent years notably in relation to the classification program of nuclear C$^*$-algebras. On the one hand because it provides a serious obstruction to the original Elliott Conjecture (see Toms \cite{tomsann}) but also, on the other hand, because its structure contains large amounts of information coming from the Elliott Invariant, and can, in many cases, be recovered from it (see for example \cite{bt}, \cite{bpt} and \cite{dad.toms}). For more details concerning the Cuntz semigroup we refer the reader to \cite{blac,kiroramer,rorfunct}.

Coward, Elliott and Ivanescu propose in \cite{CEI} a modified version of the Cuntz semigroup, $\textrm{Cu}(A)$. They use suitable equivalence classes of countably generated Hilbert modules (which, in the case of stable rank one, amount to isomorphism) to obtain a semigroup strongly related to the classical Cuntz semigroup; it is in fact isomorphic to $\W(A\otimes\mathcal{K})$. The advantage of their approach is that they further provide a category Cu for this new semigroup, consisting of positively ordered abelian semigroups with some additional properties of a topological nature. Mainly, $\textrm{Cu}(A)$ is closed under suprema of increasing sequences.

In this paper we enlarge this abstract setting to embrace both $\W(A)$ and $\textrm{Cu}(A)$. In Section~\ref{sec:precu} we build a category of ordered abelian monoids, PreCu, 
to which the original Cuntz semigroup belongs for a large class of C$^*$-algebras. This category contains Cu as a full subcategory, and differs from it in that monoids need not be closed under suprema of ascending sequences. We subsequently define, in Section~\ref{sec:complecio}, the completion of a monoid in PreCu, in terms of universal properties. Such an object is proved to exist in Section \ref{sec:existence} by providing an explicit construction. The completion obtained is thus unique and gives us a functor from the category PreCu to Cu which is a left adjoint of the identity. This, in the particular case of the Cuntz semigroup, allows us to describe $\textrm{Cu}(A)$  as a completion of $\W(A)$.
This approach is proved to be useful in computing certain Cuntz semigroups as we see in Section~\ref{sec:applications} where we recover results by Brown and Toms \cite{bt}
in which the stabilized Cuntz semigroup of certain classes of simple, unital, exact and separable C$^*$-algebras is described in terms of K-theory and traces. 
Another categorical property proved in \cite{CEI} for Cu is that it admits  countable inductive limits in such a way that the functor Cu$(-)$
from the category of C$^*$-algebras (where arbitrary inductive limits always exist) to the category Cu is (countably) continuous. 
This is not true for  the Cuntz semigroup in its classical form. The most basic counterexample is the inductive sequence of C$^*$-algebras defining the compact operators ${\mathcal K}=\lim_{\rightarrow} M_{n}(\C)$. At the level of semigroups, this induces the sequence $\N\stackrel{\mathrm{id}}{\to}\N\stackrel{\mathrm{id}}{\to}\N\to\cdots$, whose limit is not $\W(\mathcal{K})=\N\cup\{\infty\}$. 
It is natural, however, to expect continuity of the functor $\W(-)$ \emph{after} completion, in the case PreCu admits inductive limits.
As we will see in Section~\ref{sec:limits}, those do not always exist in PreCu, but this situation can be remedied by defining a smaller category $\mathcal{C}$ sitting between Cu and PreCu, to which $\W(A)$ still belongs for a large class of C$^*$-algebras. For this category, we prove that inductive limits always exist and also that we do have continuity of $\W(-)$ after completion. This, together with the preceding results, can be applied to compute the stabilized Cuntz semigroup of some direct limits of C$^*$-algebras.

\section{Notation and preliminaries}

Throughout, $M$ will denote a commutative monoid, written additively, with neutral element $0$. We shall also assume that $M$ is equipped with a partial order $\leq$ (compatible with addition) such that $x\geq 0$ for any $x$ in $M$. In particular $\leq$ will extend the algebraic order, that is, if $x+z=y$, then $x\leq y$.

All maps between monoids will be additive ordered maps that preserve $0$. Recall that a monoid map $\varphi\colon M\to N$ is an \emph{order-embedding} if $a\leq b$ whenever $\varphi(a)\leq\varphi(b)$.

Given an increasing sequence $(y_n)$ in $M$, an element $y$ is a \emph{supremum} of $(y_n)$ provided it is a least upper bound. If they exist, suprema of increasing sequences are unique. We shall denote, as is customary,  $\sup y_n$ the supremum of the increasing sequence $(y_n)$. Since our considerations might involve different monoids, we will when necessary use $\sup_M$ to mean that the supremum is computed in the monoid $M$.

\begin{definition} 
Let $x$, $y$ be elements in $M$. We write $x\ll y$ if, whenever $(y_n)$ is an increasing sequence in $M$ whose supremum exists in $M$ and $y\leq\sup(y_{n})$, then $x\leq y_{k}$ for some $k$. If $x\ll y$, we shall say that $x$ is \emph{way below} $y$. A sequence $(x_n)$ in $M$ such that $x_n\ll x_{n+1}$ for all $n$ will be called \emph{rapidly increasing}, and an element $x\in M$ such that $x\ll x$ will be called \emph{compact}. 
\end{definition}

Observe that if $x\leq y$ and $y\ll z$, then $x\ll z$. Likewise, if $x\ll y$ and $y\leq z$, then $x\ll z$.

The previous definition was given by D. Scott in \cite{scott} for general posets and was first used in connection with the Cuntz subequivalence of positive elements in C$^*$-algebras
in \cite{CEI}. We briefly recall the definitions.

\begin{definition}
{\rm cf. \cite{cu}}
Let $A$ be a C$^*$-algebra, and let $a$, $b\in A_+$. We say that $a$ is Cuntz subequivalent to $b$, in symbols $a\precsim b$, if there is a sequence $(v_n)$ in $A$ such that $a=\lim_nv_nbv_n^*$. We say that $a$ is Cuntz equivalent to $b$ if both $a\precsim b$ and $b\precsim a$ occur. Upon extending this relation to $M_{\infty}(A)_+$, one obtains an abelian semigroup $\W(A)=M_{\infty}(A)_+/\!\sim$. Denoting classes by $\langle a\rangle$, the operation and order are given by
\[
\langle a\rangle+\langle b\rangle=\langle \left(\begin{smallmatrix} a & 0 \\ 0 & b\end{smallmatrix}\right)\rangle=\langle a\oplus b\rangle\,,\,\,
\langle a\rangle\leq \langle b\rangle \text{ if } a\precsim b\,. 
\]
The semigroup $\W(A)$ is referred to as the \emph{Cuntz semigroup}.
\end{definition}

For a compact convex set $K$, we shall use $\mathrm{LAff}(K)^{++}$ to refer to those affine, lower semicontinuous functions defined on $K$, with values on $\mathbb{R}^{++}\cup\{\infty\}$. (Here, $\mathbb{R}^{++}$ stands for the strictly positive real numbers.) Note that this is a partially ordered semigroup with the usual pointwise addition and order. We shall denote by $\mathrm{LAff}_b(K)^{++}$ the subsemigroup of $\mathrm{LAff}(K)^{++}$ consisting of those functions that are \emph{bounded}. We shall also use $\mathrm{Aff}(K)^{++}$ to refer to the subsemigroup of $\mathrm{LAff}_b(K)^{++}$, whose elements are those affine, continuous and strictly positive functions defined on $K$.

\section{The category PreCu}\label{sec:precu}

We start by definining a category of monoids that will center our attention. It is modelled after the category Cu, introduced by Coward, Elliott and Ivanescu in \cite{CEI}, as an abstract setting where the Cuntz semigroup of a stable C$^*$-algebra naturally belongs to. The difference between our definition and theirs is that we do not require all increasing sequences in our monoids to have suprema.

\begin{definition}
\label{def:precu}
Let PreCu be the category defined as follows. Objects of PreCu will be partially ordered abelian monoids $M$ satisfying the properties below:
\begin{enumerate}[{\rm (i)}]
\item Every element in $M$ is a supremum of a rapidly increasing sequence.
\item The relation $\ll$ and suprema are compatible with addition.
\end{enumerate}
Maps of PreCu are monoid maps preserving
\begin{enumerate}[{\rm (i)}]
\item suprema of increasing sequences (when they exist), and
\item the relation $\ll$.
\end{enumerate}
\end{definition}

%
%
%

In view of Definition \ref{def:precu}, we may define the category Cu as follows:
\begin{definition}
\label{def:cu} {\rm (see \cite{CEI})}
Let Cu be the full subcategory of PreCu whose objects are those partially ordered abelian monoids (in PreCu) for which every increasing sequence has a supremum.
\end{definition}

As proved in \cite{CEI}, for any C$^*$-algebra $A$, the Cuntz semigroup $\W(A\otimes\mathcal{K})$ is an object of Cu. It is a natural question to ask whether the Cuntz semigroup $\W(A)$ of a C$^*$-algebra $A$ always belongs to the category PreCu. 

For a C$^*$-algebra $A$, we shall denote by $\iota\colon A\to A\otimes \mathcal{K}$ the natural inclusion defined by $\iota(a)=a\otimes e_{11}$. This map induces a map at the level of Cuntz semigroups, that we shall also denote by $\iota$, which is an order-embedding. To see this, we identify $A$ with its image inside $A\otimes \mathcal{K}$ and suppose that $\iota(\langle a\rangle)\leq\iota(\langle b\rangle)$. Then $a\precsim b$ in $(A\otimes \mathcal{K})_{+}$, and so $a\precsim b^3$,   
i.e. $a= \lim v_{n}b^3v^{*}_{n}$
for a sequence $(v_{n})$ in $A\otimes\mathcal{K}$. Right and left multiplication by $a$ implies that $a^3= \lim (av_{n}b)b(bv^{*}_{n}a)$ where now $av_{n}b$ belongs to $A$ since the latter is a hereditary C$^*$-subalgebra of $A\otimes\mathcal{K}$. Therefore $\langle a^{3}\rangle
\leq\langle b\rangle $ in $\W(A)$, and so $\langle a\rangle \leq
\langle b\rangle$ in $\W(A)$.

We have that any $a\in M_{\infty}(A)_+$ is the limit, in norm, of the increasing sequence $(a-1/n)_+$, so that indeed $\langle a\rangle =\sup\langle (a-1/n)_+\rangle$. But, while $\langle (a-\epsilon)_+\rangle\ll \langle a\rangle$ in $\W(A\otimes\mathcal K)$, it is not obvious this is the case anymore in $\W(A)$. We have the following
\begin{lemma}
\label{lem:waybelowvssuprema}
Let $A$ be a C$^*$-algebra. The following conditions are equivalent:
\begin{enumerate}[{\rm (i)}]
\item $\langle(a-\epsilon)_+ \rangle\ll\langle a\rangle$ for any $\epsilon>0$ and any $a\in M_{\infty}(A)_+$.
\item $\sup_{\W(A)}x_n=\sup_{\W(A\otimes\mathcal K)}x_n$ whenever $(x_n)$ is an increasing sequence in $\W(A)$ with supremum in $\W(A)$.
\end{enumerate}
\end{lemma}
\begin{proof}
(i) $\Rightarrow$ (ii):  Let $(x_n)$ be an increasing sequence in $\W(A)$ with $x=\sup_{\W(A)}x_n$ and $y=\sup_{\W(A\otimes \mathcal K)}x_n$. Clearly $y\leq x$. If we write now $x=\langle a\rangle$ for some $a\in M_{\infty}(A)_+$, our assumption implies that, for any $n$, there exists $m$ with $\langle (a-1/n)_+\rangle\leq x_m$. Thus $\langle (a-1/n)_+\rangle\leq y$ for every $n$, whence $x=\langle a\rangle\leq y$.

(ii) $\Rightarrow$ (i): Let $a\in M_{\infty}(A)_+$, and suppose that $x_n$ is an increasing sequence in $\W(A)$ with supremum in $\W(A)$. The assumption implies that this agrees with the supremum in $\W(A\otimes\mathcal K)$. If $\langle a\rangle\leq \sup x_n$, then as $\langle(a-\epsilon)_+ \rangle\ll\langle a\rangle$ in $\W(A\otimes\mathcal K)$ (by the results in \cite{CEI}), we have $\langle(a-\epsilon)_+ \rangle\leq x_n$ for some $n$ in $\W(A\otimes\mathcal K)$, hence also in $\W(A)$.
\end{proof}


\begin{definition}
\label{dfn:hereditary}
Let $M$ and $N$ be partially ordered monoids. An order-embedding $f\colon M\to N$ will be called \emph{hereditary} if, whenever  $x\in N$ and $y\in f(M)$ satisfy $x\leq y$, it follows that $x\in f(M)$.
\end{definition}

\begin{lemma}\label{lem:hereditary}
Let $N$ be an object of Cu and $M$ be a partially ordered monoid. Let $f\colon M\to N$ be a hereditary map. Then, for any increasing sequence $(x_{n})$ in $M$ with supremum $x$ in $M$, we have $f(x)=\sup(f(x_{n}))$
\end{lemma}
\begin{proof}
Since $x_{n}\leq x$ for all $n$, it follows that $f(x_{n})\leq f(x)$ for all $n$. Therefore 
$\sup(f(x_{n}))\leq f(x)$. Our assumption on $f$ now implies that $\sup(f(x_{n}))=f(y)$  for some $y\in M$. Using that $f$ is an order-embedding we obtain that $x_n\leq y\leq x$ for all $n$, so $y=x$.
\end{proof}

\begin{definition}
Let $A$ be a C$^*$-algebra. We will say that $\W(A)$ is \emph{hereditary} if the map $\iota\colon \W(A)\to \W(A\otimes\mathcal{K})$ is hereditary.
\end{definition}

\begin{lemma}
\label{lem:hervssup}
Let $A$ be a C$^*$-algebra such that $\W(A)$ is hereditary. Then, given $a\in M_{\infty}(A)_+$ and $\epsilon>0$, we have $\langle (a-\epsilon)_+\rangle\ll\langle a\rangle$ in $\W(A)$. 
\end{lemma}
\begin{proof}
Since $\W(A\otimes\mathcal K)$ is an object of Cu, and $\W(A)$ is hereditary, we may apply Lemma \ref{lem:hereditary} to conclude that suprema of increasing sequences in $\W(A)$ agree, when they exist, with suprema in $\W(A\otimes\mathcal K)$. The conclusion now follows from Lemma \ref{lem:waybelowvssuprema}.
\end{proof}

\begin{proposition}
\label{prop:wainprecu}
Let $A$ be a C$^*$-algebra such that $\W(A)$ is hereditary. Then $\W(A)$ is an object of PreCu and the map $\iota\colon\W(A)\to\W(A\otimes\mathcal{K})$ is an order-embedding in PreCu.
\end{proposition}
\begin{proof}
By Lemma \ref{lem:hereditary}, coupled with Lemma \ref{lem:waybelowvssuprema}, we have that every element $\langle a\rangle$ in $\W(A)$ is the supremum of the rapidly increasing sequence $\langle (a-1/n)_+\rangle$.

Assume now that $x\ll y$ and $z\ll t$ in $\W(A)$, and write $y=\langle a\rangle$ and $t=\langle b\rangle$. It follows that there is $n$ with $x\leq \langle (a-1/n)_+\rangle$ and  $z\leq \langle (b-1/n)_+\rangle$. Therefore, since 
\[
(a-1/n)_+\oplus (b-1/n)_+\sim (a\oplus b-1/n)_+\,,
\]
we get 
\[
x+z\leq \langle (a\oplus b-1/n)_+\rangle\ll \langle a\rangle+\langle b\rangle\,.
\]

Finally, suppose that $(x_n)$ and $(y_n)$ are increasing sequences in $\W(A)$ with suprema $x$ and $y$ respectively. We are to show that $x+y$ is the supremum of $(x_n+y_n)$. Using that $\W(A)$ is hereditary, one sees that $(x_n+y_n)$ also has a supremum in $\W(A)$. And using that suprema and addition are compatible in $\W(A\otimes\mathcal{K})$, we see, using Lemma \ref{lem:waybelowvssuprema}, that
\[
\sup_{\W(A)}x_n+\sup_{\W(A)}y_n=\sup_{\W(A\otimes\mathcal{K})}x_n+\sup_{\W(A\otimes\mathcal{K})}y_n=\sup_{\W(A\otimes\mathcal{K})}(x_n+y_n)=\sup_{\W(A)}(x_n+y_n)\,.
\]
We have already observed that $\iota$ is an order-embedding. Combining again Lemma \ref{lem:hereditary} and Lemma \ref{lem:waybelowvssuprema}, we see that it preserves suprema of increasing sequences.

Suppose now that $\langle a\rangle\ll\langle b\rangle$. Since $\langle b\rangle =\sup(\langle
  (b-\varepsilon)_{+}\rangle)$ in $\W(A)$, there exists $\varepsilon >0$ such that $\langle a\rangle \leq \langle (b-\varepsilon)_{+}\rangle$. Applying $\iota$,
\[
\iota(\langle a\rangle)\leq\iota(\langle (b-\varepsilon)_{+}\rangle)=\langle( b-\varepsilon)_{+})\rangle\ll\langle b\rangle=\iota(\langle b\rangle)\text{ whence }\iota(\langle a\rangle)\ll\iota(\langle b\rangle)\,,
\]
so that $\iota$ is a map in PreCu.
\end{proof}

\begin{remark}\label{rem:Anohereditaria} There are examples of C$^*$-algebras $A$ for which $\W(A)$ is not hereditary (\cite{betal}).
The Cuntz semigroup of these algebras contains bounded ascending sequences with no least upper bound. This is needed for $\W(A)$ to be hereditary (see Proposition~\ref{prop:heriffC}),
but not for $\W(A)$ to be in PreCu.
\end{remark}

\begin{lemma}
\label{lem:sr1her}
Let $A$ be a C$^*$-algebra with $\mathrm{sr}(A)=1$. Then $\W(A)$ is hereditary. 
\end{lemma}
\begin{proof}
Let $a\in A\otimes\mathcal{K}_+$ and $b\in M_{\infty}(A)_+$, and assume that $a\lesssim b$.  We are to show that there exists $c\in M_{\infty}(A)_+$ such that $c\sim a$. 

Using that $A\otimes\mathcal{K}$ is the completion of $M_{\infty}(A)$, we obtain that if $a\in (A\otimes\mathcal{K})_{+}$, there exists a sequence $(a_{n})$ belonging 
to $M_{\infty}(A)_{+}$ such that $a=\lim (a_{n})$, and we may assume that $\|a-a_{n}\|\leq 1/n$.
By Lemma 2.2 in \cite{kiror}, there are contractions $d_n$ in $A\otimes\mathcal{K}$ such that $(a-1/n)_{+}=d_{n}a_{n}d_{n}^{*}$. Thus  
\[
(a-1/n)_{+}=d_{n}a_{n}d_{n}^{*}\sim a^{1/2}_{n}d^{*}_{n}d_{n}a^{1/2}_{n}\,,
\]
and $b_n:=a_{n}^{1/2}d^{*}_{n}d_{n}a_{n}^{1/2}\in M_{\infty}(A)_{+}$. This implies that $\langle a\rangle=\sup \langle b_n\rangle$ in $\W(A\otimes\mathcal{K})$ and that the sequence $(\langle b_n\rangle)$ is rapidly increasing in $W(A\otimes\mathcal{K})$.

Notice that the sequence $(\langle b_n\rangle)$ is bounded above in $\W(A)$ by $\langle b\rangle$. Therefore, it also has a supremum $\langle c\rangle$ in $\W(A)$, by \cite[Lemma 4.3]{bpt}.  
The arguments in \cite{bpt} show that for each $n$ there exists $m$ and $\delta_n>0$ with $\delta_n\to 0$ such that $(c-1/n)_{+}\precsim (b_{m}-\delta_{n})_{+}$. This implies then that
\[
(c-1/n)_{+}\precsim (b_{m}-\delta_{n})_{+}\precsim b_{m}\precsim a
\]
in $A\otimes\mathcal{K}$, whence $c\precsim a$.

On the other hand, since clearly $b_{n}\precsim c$ for all $n$, and $\langle a\rangle$ is the supremum in $\W(A\otimes\mathcal{K})$ of $\langle b_n\rangle$, we see that $a\precsim c$. Thus $c\sim a$.
\end{proof}

For a unital C$^*$-algebra $A$, we denote as usual by $\T(A)$ the simplex of normalized traces. Given a trace $\tau\in\T(A)$ and $a\in M_{\infty}(A)_+$, we may construct $d_{\tau}(a)=\lim_n\tau(a^{1/n})$. It turns out that 
\[
d_{\tau}\colon M_{\infty}(A)_+\to\mathbb{R}^+\,,\, a\mapsto d_{\tau}(a)
\]
is lower semicontinuous and does not depend on the Cuntz class of $a$. Thus, it defines a state on $\W(A)$, termed a \emph{lower semicontinuous dimension function}.
Let us denote by $LDF(A)$ the set of all lower semicontinuous dimension functions on $A$. Observe also that the function $\hat{a}\colon \T(A)\to\mathbb{R^+}$ defined by $\hat{a}=(\tau)=d_{\tau}(a)$ is an element of $\mathrm{LAff}(\T(A))^+$.

In the case that $A$ is moreover simple, we say that $A$ has \emph{strict comparison} if the order in $\W(A)$ is determined by lower semicontinuous dimension functions. Namely, if $d_{\tau}(a)<d_{\tau}(b)$ for all $\tau\in \T(A)$, it follows that $a\precsim b$. 

\begin{definition}
\label{def:rc} {\rm (\cite{tomsjfa,tomsplms})} A unital C$^*$-algebra $A$ has \emph{$r$-comparison} if whenever $a$, $b\in M_{\infty}(A)_+$ satisfy
\[
s(\langle a\rangle)+r<s(\langle b\rangle)\,,\text{ for every $s\in\mathrm{LDF}(A)$}\,,
\]
then $\langle a\rangle\leq\langle b\rangle$. The \emph{radius of comparison} of $A$ is
\[
\mathrm{rc}(A)=\inf\{r\geq 0\mid A \text{ has $r$-comparison}\}\,,
\]
which is understood to be $\infty$ if the infimum does not exist. 
\end{definition}

\begin{theorem}
\label{thm:rcher} {\rm (\cite{betal})} Let $A$ be a unital C$^*$-algebra with finite radius of comparison. Then $\W(A)$ is hereditary.
\end{theorem}

\section{Abstract completions}\label{sec:complecio}

The purpose of this section is to define the notion of completion for an object of PreCu and establish the universal property it satisfies, from which uniqueness of this object will follow. Existence will be proved by constructing a concrete completion and will be carried out in the next section. We also show that the completion process induces a left adjoint functor of the identity.

\begin{definition}
\label{dfn:completion}
Let $M$ be an object of PreCu. We say that a pair $(N,\iota)$ is a \emph{completion} of $M$ if
\begin{enumerate}[{\rm (i)}] 
\item $N$ is an object of Cu, 
\item $\iota\colon M\to N$ is an order-embedding in PreCu, and
\item for any $x\in N$, there is a rapidly increasing sequence $(x_{n})$ in $M$ such that $x=\sup\iota(x_{n})$.
\end{enumerate}
\end{definition}

\begin{lemma}\label{orderembedding}
Let $M$, $N$ be objects of PreCu and let $\alpha\colon M\to N$ be a map in PreCu. The following conditions are equivalent:
\begin{enumerate}[{\rm (i)}]
\item  $\alpha$ is an order-embedding.
\item $\alpha (x)\ll\alpha(y)$ 
if and only if $x\ll y$. 
\end{enumerate}
\end{lemma}
\begin{proof}
(i) $\Rightarrow$ (ii): If $x\ll y$, then as $\alpha$ is a map in PreCu, it follows that $\alpha (x)\ll\alpha(y)$. Conversely, suppose that $\alpha(x)\ll\alpha(y)$. If $y\leq \sup(z_{n})$, then $\alpha(y)\leq \alpha(\sup(z_n))=\sup\alpha(z_n)$, so that there is $n$  with $\alpha(x)\leq \alpha(z_{n})$. Since $\alpha$ is order-embedding, this implies $x\leq z_{n}$ .

(ii) $\Rightarrow$ (i): Suppose now that $\alpha(x)\leq\alpha(y)$. Write $x=\sup(x_{n})$, where $x_{n}\ll x_{n+1}$. Since $\alpha$ is a map in PreCu, $\alpha(x)=\sup(\alpha(x_{n}))$ and $\alpha(x_{n})\ll\alpha(x_{n+1})$. Therefore, $\alpha(x_{n})\ll\alpha(x_{n+1})\leq\alpha(y)$. Using the hyphotesis we obtain that $x_{n}\ll y$ for all $n$, and then $x\leq y$.
\end{proof}

\begin{theorem}\label{universality}
Let $M$ be an object of PreCu, $P$ an object of $Cu$ and $\alpha\colon M\to P$ a map in PreCu. Then, if $(N,\iota)$ is a completion of $M$, there exists a unique map $\beta\colon N\to P$ in Cu such that $\beta\circ\iota=\alpha$. Moreover, if $\alpha$ is an order-embedding then so is $\beta$.
\end{theorem}

\begin{proof}
Since we can write any $x\in N$ as $x=\sup\iota (x_{n})$, where $(x_{n})$ is a rapidly increasing sequence in $M$, define $\beta\colon N\to P$ by $\beta(x)=\sup(\alpha(x_{n}))$. We need to check that $\beta$ is well defined. 

Suppose that $x=\sup\iota(x_{n})=\sup \iota(y_{m})$ where $(x_{n})$ and $(y_{m})$ are both rapidly
increasing sequences in $M$. Then, for every $n$, there exist $m$ and $k$ such that  
\[
x_{n}\ll y_{m}\ll x_{k}\,
\]
Since $\alpha$ is a map in PreCu, we obtain that 
\[
\alpha(x_{n})\ll \alpha(y_{m})\ll \alpha(x_{k})\,.
\]
Therefore $\sup \alpha(x_{n})=\sup \alpha(y_{n})$, whence $\beta$ is well-defined.

That $\beta$ is additive and preserves the identity element follows easily from the fact that $\alpha$ belongs to PreCu.

Next, let $x\leq y$ in $N$. Write $x=\sup(\iota(x_{n}))$ and $y=\sup(\iota(y_{n}))$, where $(x_n)$ and $(y_n)$ are rapidly increasing. As before we obtain that for every $n$, there is $m$ with $x_{n}\ll y_{m}$, so then
\[
\alpha(x_{n})\ll \alpha(y_{m})\leq \sup\alpha(y_k)=\beta(y)\,.
\]
Therefore $\beta(x)=\sup(\alpha(x_{n}))\leq\beta(y)$, hence $\beta$ is order-preserving.

If now $x=\sup (x_n)\in M$, with $(x_{n})$ rapidly increasing, apply $\iota$ followed by $\beta$ so that
\[
\beta(\iota(x))=\beta(\iota(\sup(x_{n})))=\beta(\sup(\iota(x_{n})))=\sup(\alpha(x_{n}))=\alpha(x)\,,
\]
which shows $\beta\circ\iota=\alpha$.

Suppose that $x\ll y$, for elements $x$, $y$ in $N$. Write $y=\sup\iota(y_n)$, where $(y_n)$ is a rapidly increasing sequence in $M$. There exists then $n$ such that $x\leq \iota (y_n)$, and since $y_n\ll y_{n+1}$, we may apply $\alpha$ to obtain $\alpha(y_n)\ll \alpha(y_{n+1})$. Now
\[
\beta(x)\leq\beta(\iota(y_n))=\alpha(y_n)\ll\alpha(y_{n+1})=\beta(\iota(y_{n+1}))\leq\beta(y)\,,
\]
hence $\beta (x)\ll\beta (y)$.

It remains to be shown that $\beta$ preserves suprema. Let $\{x_{n}\}$ be an increasing sequence in $N$ and let $x=\sup(x_{n})$. As $\beta$ is order-preserving we readily get $\beta(x_{n})\leq \beta(x)$ for all $n$. Therefore $\sup(\beta(x_{n}))\leq \beta(x)$.
Since we can also write $x=\sup\iota(y_{n})$ (for a rapidly increasing sequence $(y_n)$) it follows that, given $n$, there is $m$ with $\iota(y_{n})\leq x_{m}$. Apply $\beta$ to get 
$\alpha(y_n)=(\beta\circ\iota)(y_{n})\leq \beta(x_{m})\leq \sup(\beta(x_{k}))$, whence 
\[
\beta(x)=\sup(\alpha(y_{n}))\leq \beta(x_{m})\leq \sup(\beta(x_{k}))\,. 
\]

We now prove that $\beta$ is unique. To this end, assume $\beta'\colon N\to P$ is another map in $Cu$ such that $\beta'\circ\iota=\alpha=\beta\circ\iota$. Let $x=\sup(\iota(x_{n}))$
be an element in $N$. Then
\[
\beta'(x)=\beta'(\sup\iota(x_{n}))=\sup((\beta'\circ\iota)(x_{n}))=\sup((\beta\circ\iota)(x_{n}))=
\beta(\sup\iota(x_{n}))=\beta(x)\,.
\]
Assume finally that $\alpha$ is an order-embedding and suppose that $\beta(x)\leq\beta(y)$. With $x=\sup\iota(x_{n})$, and $y=\sup\iota(y_{n})$ for rapidly increasing sequences $(x_n)$ and $(y_n)$, this implies that, for any $n$,
\[
\alpha(x_n)\ll\alpha(x_{n+1})\leq\sup\alpha(x_n)\leq\sup\alpha(y_n)\,. 
\]
There is then $m$ depending on $n$ with $\alpha(x_n)\leq\alpha(y_m)$, that is, $x_n\leq y_m$ (as $\alpha$ is an order-embedding). Thus $\iota(x_{n})\leq \iota(y_{m})\leq y$, and so $x\leq y$. This shows that $\beta$ is an order-embedding.
\end{proof}
A standard argument yields the following.
\begin{corollary}
Let $M$ be an object in PreCu. If there is a completion of $M$, then it is unique (up to order-isomorphism).
\end{corollary}

In view of the corollary above, given an object $M$ of PreCu, we shall write $(\overline{M},\iota)$ to refer to its (unique) completion, in case it exists.

\begin{proposition}
If every object of PreCu has a completion $(\overline{M},\iota)$, then the map $\iota\colon M\to\overline{M}$ induces a covariant functor
$C\colon PreCu\to Cu$, which is a left adjoint of the identity, i.e., if $i\colon Cu\to PreCu$ is the inclusion functor, then $C\circ i\sim id_{Cu}$ (where $\sim$ denotes natural equivalence).
\end{proposition}
\begin{proof}
Put 
\[
\begin{tabular}{c c c c}
  $C : $ & PreCu &$\longrightarrow$ & Cu \\
      & $M$& $\mapsto$  & $\overline{M}$ \\

\end{tabular}\\\,.
\]
Let $\sigma\colon M\to N$ be a map in PreCu, and let $(\overline{M},\iota)$ and $(\overline{N},\iota_{1})$ be the completions of $M$ and $N$ respectively. 

Now use that the map $\iota_{1}\circ\sigma:M\to\overline{N}$ belongs to PreCu and Theorem \ref{universality}, so that there is a unique morphism $\overline{\sigma}\colon\overline{M}\to\overline{N}$ in Cu such that $\overline{\sigma}\circ\iota=\iota_{1}\circ\sigma$. In fact, if $x=\sup(\iota(x_{n}))$ is an element in $\overline{M}$ (with $(x_n)$ rapidly increasing in $M$), then
\[
\overline{\sigma}(x)=\sup((\iota_{1}\circ\sigma)(x_{n}))\,.
\]

It is a simple matter to check that $\overline\sigma$ satisfies $\overline{\mathrm{id}_M}=\mathrm{id}_{\overline{M}}$ and that $\overline{\sigma\circ\tau}=\overline{\sigma}\circ\overline{\tau}$.

Finally, notice that if $M$ is already an object of Cu, then $(M,\mathrm{id})$ is a completion, hence unique, whence $C\circ i(M)\cong M$.
\end{proof}
\begin{remark}
It follows from the result above that, if every object in PreCu admits a completion and $M$ is in PreCu, then $\overline{\overline M}\cong \overline M$.
\end{remark}

\section{Existence of completions in PreCu}\label{sec:existence}

In order to construct the completion of an object of PreCu, we basically need to add suprema of every ascending sequence of the given object. This is best captured by using intervals -- we recall the definitions below. 

Let $M$ be a partially ordered abelian monoid. An \emph{interval} in $M$ is a nonempty subset $I\subseteq M$ which is upward directed and order-hereditary (i.e. if $x\leq y$ and $y\in I$, then $x\in I$). Denote $\Lambda(M)$ the set of all intervals in $M$. Note that $\Lambda(M)$ is equipped with a natural abelian monoid structure, namely if $I,J\in\Lambda(M)$, then
\[
I+J=\{z\in M\mid z\leq x+y \text{ for some } x\in I, y\in J\}\,.
\]

We shall be exclusively concerned with \emph{countably generated} intervals. Those are elements $I$ in $\Lambda (M)$ that have a countable cofinal subset, that is, a countable subset $X$ -- that may always be assumed to be upwards directed -- such that  $I=\{x\in M\mid x\leq y \text{ for some }y\in X\}$. In that case, there is an increasing sequence $(x_n)$ in $M$ with $I=\{x\in M\mid x\leq x_n\text{ for some }n \}$. We will denote by $\Lambda_{\sigma}(M)$ the monoid of countably generated intervals over M. We shall sometimes refer to a rapidly increasing cofinal subset $X$ for an interval $I$, meaning that for any $x$, $y\in X$, there is $z\in X$ with $x$, $y\ll z$. Note that an interval $I$ with rapidly increasing countable cofinal subset may be then written as $I=\{x\in M\mid x\leq x_n\text{ for some }n\}$ where now $(x_n)$ is a rapidly increasing sequence.


For intervals $I$ and $J$ in $\Lambda_{\sigma}(M)$, write $I\precsim J$ if given $x$ in $I$ and $z\ll x$, there exists $y\in J$ such that $z\ll y$. It is 
clear that this is equivalent to taking $x$ and $y$ above in countable cofinal subsets $X$ and $Y$ for $I$ and $J$ respectively. The relation $\precsim$ is easily seen to be reflexive and transitive, and induces a congruence on $\Lambda_{\sigma}(M)$ by defining $I\sim J$ if $I\precsim J$ and $J\precsim I$.

Put $\overline{M}=\Lambda_{\sigma}(M)/\sim$, and denote the elements of $\overline{M}$ by $[I]$.

\begin{proposition}\label{propo}
Let $M$ be an object in PreCu. Then $\overline{M}$ is a partially ordered monoid.
\end{proposition}

\begin{proof}
Define $[I]\leq [J]$ if $I\precsim J$ in $\Lambda_{\sigma}(M)$, which is clearly well-defined, so that $\overline{M}$ becomes a partially ordered set.

Now, let $I$, $J$, $K$, $L\in\Lambda_{\sigma}(M)$ and assume that $[I]\leq [J]$ and $[K]\leq[L]$ in $\overline{M}$. Choose increasing countable cofinal sets $\{x_n\}$, $\{y_n\}$, $\{z_n\}$ and $\{t_n\}$ for $I$, $J$, $K$ and $L$ respectively. Write each term in the above sequences as a supremum of rapidly increasing sequences, namely: $x_n=\sup_m x_n^m$, $y_n=\sup_m y_n^m$, $z_n=\sup_m z_n^m$ and $t_n=\sup_m t_n^m$.

If $x\ll x_{n}+z_{n}$, then there exists $m$ such that $x\leq x^{m}_{n}+z^{m}_{n}\ll x^{m+1}_{n}+z^{m+1}_{n}$. Since $x^{m+1}_{n}\ll x_{n}$ and $z^{m+1}_{n}\ll z_{n}$, we see that  $x^{m+1}_{n}\ll y_{p}$ and  $z^{m+1}_{n}\ll t_{q}$ for some $p$, $q$. As addition and $\ll$ are compatible, it follows that $x \leq x^{m}_{n}+z^{m}_{n}\ll y_{t}+t_{k}$ where $k=\max\{p,q\}$. Therefore $[I+K]\leq [J+L]$, and we may define $[I]+[J]=[I+J]$. This operation is clearly compatible with the order.
\end{proof}

Let $M$ be an object of PreCu. For any $x\in M$, put $I(x)=[0,x]=\{ y\in M\mid y\leq x\}$, which is clearly a countably generated interval. Next, define a map $\iota$ as follows:
\[
\begin{tabular}{ c c c }
$\iota\colon M$ & $\longrightarrow$ & $\overline{M}$\\
$\,\,\, x$ & $\mapsto$ & $[I(x)]$\,. \\
\end{tabular}
\] 
Observe that $\iota$ is an order-embedding. Indeed, $\iota$ is additive and preserves the order. If now $[I(x)]\leq [I(y)]$ in $\overline{M}$, then $I(x)\precsim I(y)$ in $\Lambda_{\sigma}(M)$. Write $x=\sup(x_{n})$, where $(x_{n})$ is rapidly increasing. Since $x_{n}\ll x$ for all $n$ we have $x_{n}\leq y$ for all $n$, whence $x\leq y$

%

\begin{lemma}
\label{lem:rapincr}
Let $M$ be an object of PreCu. Then every element in $\Lambda_{\sigma}(M)$ is equivalent to an interval with a rapidly increasing countable cofinal subset.
\end{lemma}

\begin{proof}
Let $(x_n)$ be an increasing countable cofinal subset for $I\in  \Lambda_{\sigma}(M)$. For each $n$ choose a rapidly increasing sequence $(x_{n}^m)$ with $x_{n}=\sup(x^{m}_{n})$, and then consider the rapidly increasing cofinal subset $(x_n^m)$ now varying $n$ and $m$. Let $J$ be the interval generated by $(x_n^m)$. It is clear that $J\subseteq I$ and therefore $J\precsim I$. Consider $x\in M$ such that $x\ll x_n$. Since $x_n=\sup x_n^m$, there exists $k$ such that $x\ll x_n^k$. Therefore $I\precsim J$.
%
%
%
\end{proof}

\begin{proposition}\label{suprem}
Let $M$ be an object in PreCu. Then, every increasing sequence in $\overline{M}$ has a supremum in $\overline{M}$. More concretely, if $([I_n])$ is an increasing sequence in $\overline{M}$,  and $X_n$ is a rapidly increasing countable cofinal subset for each $I_n$, then $X=\cup_n X_n$ is a countable cofinal subset for $\sup([I_n])$.
\end{proposition}
\begin{proof}
Let $([I_{n}])$ be an increasing sequence in $\overline{M}$, and let $X_n$ be rapidly increasing countable cofinal subsets for each $I_n$ (this is no loss of generality, in view of Lemma \ref{lem:rapincr}, to assume that $X_n$ is rapidly increasing for every $n$). Let $X=\cup_n X_n$ and put $I=\{x\in M\mid x\leq y\text{ for some }y\in X\}$. We are to show that $I\in\Lambda_{\sigma}(M)$ and that $[I]=\sup([I_n])$.

Let $x_1$, $x_2\in I$. Then there are $n$ and $m$ and $y_1\in X_n$ and $y_2\in X_m$ with $x_i\leq y_i$. We may assume that $n\leq m$. Find $y$ in $X_n$ such that $y_1\ll y$, and since $I_n\precsim I_m$, there is $z$ in $X_m$ with $y_1\ll z$. Since $y_2\in X_m$, there exists $w$ in $X_m$ with $z$, $y_2\ll w$. Then $x_i\leq y_i\ll w$ and $w\in X$. This shows that $I$ is upwards directed. As $I$ is clearly hereditary and $X$ is countable (and nonempty), we conclude that $I\in\Lambda_{\sigma}(M)$.

That $I_n\precsim I$ for all $n$ is clear. Suppose that $J\in \Lambda_{\sigma}(M)$ satisfies $I_n\precsim J$. Choose a countable cofinal subset $Y$ for $J$, and let $x\in X$ and $z\ll x$. There is $n$ with $x\in X_n$, so there is $y\in Y$ with $z\ll y$, and this implies $I\precsim J$. Therefore $[I]=\sup [I_n]$.



\end{proof}

\begin{lemma}
\label{lem:natural}
Let $M$ be an object of PreCu. If $x\ll y$ in $M$, then $[I(x)]\ll [I(y)]$ in $\overline M$.
\end{lemma}
\begin{proof}
Let $[I_n]$ be increasing in $\overline{M}$. Choose a rapidly increasing countable cofinal set $X_n$ for $I_n$ and put $X=\cup_n X_n$. If we let $I=\{x\in M \mid x\leq y\text{ for some }y\in X\}$, we know that $[I]=\sup[I_n]$.

Suppose that $[I(y)]\leq [I]$. Since $x\ll y$, there is $z$ in $X$ with $x\ll z$, and $z\in X_n$ for some $n$. Therefore, there is $n$ such that $w\ll z$ whenever $w\ll x$, which shows that $[I(x)]\leq [I_n]$.
\end{proof}

\begin{proposition}\label{rapidament}
Let $M$ be an object in PreCu. Then, every element in $\overline{M}$ is the supremum of a rapidly increasing sequence coming from $M$.
\end{proposition}
\begin{proof}
Let $[I]\in\overline{M}$. We may assume that there is a rapidly increasing sequence $(x_n)$ such that $I=\{x\in M\mid x\leq x_n\text{ for some }n\}$. Consider the sequence $([I^n])$ in $\overline{M}$, where
\[
I^{n}=\{y\in M\mid y\leq x_{n}\}\,.
\]
It follows from Lemma \ref{lem:natural} that $[I^n]\ll [I^{n+1}]$. Thus, to prove the desired result, it suffices to show that $[I]=\sup([I^{n}])$. It is clear from the definition that $[I^{n}]\leq [I]$ for all $n$.
Suppose that $J\in \Lambda_{\sigma}(M)$ has an increasing countable cofinal subset $(y_{m})$ and that $[I^n]\leq [J]$ for all $n$. Given $n$ and $x\ll x_{n}$, we can find $m$ such that $x\ll y_{m}$. Therefore $[I]\leq [J]$, obtaining that $[I]$ is the smallest upper bound of the sequence.
\end{proof}

\begin{remark}
\label{rem:rapidlyincreasing} {\rm The proof in Proposition \ref{rapidament} shows that, given $I$ and $J\in\Lambda_{\sigma}(M)$ with respective rapidly increasing countable cofinal subsets $(x_n)$ and $(y_n)$, then, letting $I^n=[0,x_n]$ and $J^n=[0,y_n]$, we have $[I^n+J^n]$ is rapidly increasing and $\sup ([I^n]+[J^n])=I+J$.}
\end{remark}

\begin{proposition}
Let $M$ be an object of PreCu. Then suprema and $\ll$ are compatible with addition in $\overline{M}$.
\end{proposition}
\begin{proof}
Let $([I_{n}])$, $([J_{n}])$ be two increasing sequences in $\overline{M}$, and let
\[
[I]=\sup([I_{n}])\text{  and  }[J]=\sup([J_{n}])\,.
\]
Choose rapidly increasing sequences $([I^n])$ and $([J^n])$ such that $[I+J]=\sup ([I^n+J^n])$, and $[I]=\sup [I^n]$, $[J]=\sup [J^n]$ (see Remark \ref{rem:rapidlyincreasing}). We can then find $m$ with $[I^n]\leq [I_m]$ and $[J^n]\leq [J_m]$. Thus:
\[
[I^{n}+J^{n}]=[I^{n}]+[J^{n}]\leq [I_{m}]+[J_{m}]\leq [I]+[J]\,,
\]
whence
\[
[I]+[J]=\sup ([I^n]+[J^n])\leq\sup([I_{n}+J_{n}])\leq [I]+[J]\,.
\]

To prove that $\ll$ and addition are compatible we have to check that if $[I]\ll [J]$ and $[K]\ll [L]$ then $[I+K]\ll [J+L]$. If we write $[J]=\sup([J^{n}])$ and $[L]=\sup([L^{n}])$, where $J^n$ and $L^n$ are constructed as in Proposition \ref{rapidament} (see also Remark \ref{rem:rapidlyincreasing}), so that the corresponding rapidly increasing sequence for $[J+L]$ will be $[J^n+L^n]$. Using that  $[I]\ll [J]$ and $[K]\ll [L]$, we find $m$ with $[I]\leq [J^{m}]$ and $[K]\leq[L^{m}]$. Therefore
\[
[I]+[K]\leq [J^{m}+L^{m}]\ll [J^{m+1}+L^{m+1}]\leq [J]+[L]\,,
\]
as desired.
\end{proof}

Collecting the results above we, we obtain:

\begin{theorem}\label{thm:completion} Let $M$ be an object of PreCu. Then $\overline{M}$ is an object of Cu,
\[\begin{tabular}{ c c c }
$\iota\colon M$ & $\longrightarrow$ & $\overline{M}$\\
$\,\,\, x$ & $\mapsto$ & $[I(x)]$, \\
\end{tabular}
\] (where $I(x)=[0,x]$) is an order-embedding in PreCu
and $(\overline{M},\iota)$ is the completion of $M$.
\end{theorem}

\begin{proof}
It is clear by the preceding results that $\overline M$ is an object of Cu, and we already know that $\iota$ is an order-embedding that preserves $\ll$ (see Lemma \ref{lem:natural}).

Let $(x_n)$ be an increasing sequence in $M$ with $x=\sup(x_{n})$ in $M$, and let $I=\{x\in M\mid x\leq x_n\text{ for some }n\}$. We first show that $\iota(x)=[I]$. It is clear that $I\precsim I(x)$. Write $x=\sup z_n$, where $(z_n)$ is rapidly increasing in $M$. Thus, if $y\ll x$, then there is $n$ such that $y\ll z_n$, so that $y\leq x_m$ for some $m$. This shows that $I(x)\precsim I$.

Now, let $J\in\Lambda_{\sigma}(M)$ and choose an increasing countable cofinal subset $(y_n)$ for $J$. If $[0,x_n]\precsim J$ for all $n$, we get that for any $n$ and $x\ll x_n$, there is $m$ with $x\ll y_m$, so that $I\precsim J$. It follows that $\iota(x)=\sup\iota(x_n)$. 

This, together with Proposition~\ref{rapidament}, shows that $(\overline{M},\iota)$ is the completion of $M$.
\end{proof}

\begin{corollary}
\label{cor:completion}
Let $M$ be an object of PreCu such that every element is compact. Then $\overline{M}\cong\Lambda_{\sigma}(M)$.
\end{corollary}
\begin{proof}
This follows directly from Theorem \ref{thm:completion} and the fact that, if every element in $M$ is compact, the relation $\precsim$ defined in $\Lambda_{\sigma}(M)$ reduces to inclusion.
\end{proof}

\section{Some applications}\label{sec:applications}

Our first application of Theorem \ref{thm:completion} relates the Cuntz semigroup of $A$ with that of $A\otimes\mathcal K$ in the case that $\W(A)$ is hereditary. 

\begin{theorem}
\label{thm:wacompletion}
Let $A$ be a C$^*$-algebra with $\W(A)$ hereditary. Then the pair $\W(A)$ is in PreCu and $(\W(A\otimes\mathcal{K}),\iota)$ is the completion of $\W(A)$.
\end{theorem}
\begin{proof}
We need to verify that $\W(A)$ and $(\W(A\otimes\mathcal{K}),\iota)$ satisfy the conditions of Definition \ref{dfn:completion}, and then invoke Theorem \ref{thm:completion}. 

We have already proved in Proposition \ref{prop:wainprecu} that, under our assumptions, $\W(A)$ is an object of PreCu and that $\iota$ is an order-embedding in PreCu. 

Also, as in the proof of Lemma~\ref{lem:sr1her}, if $a\in (A\otimes\mathcal{K})_{+}$, we obtain a sequence $\langle b_n\rangle$  with $b_n \in M_{\infty}(A)$  which is rapidly increasing 
(both in $\W(A)$ and $\W(A\otimes\mathcal{K})$ since $\W(A)$ is hereditary) and such that $\langle a\rangle=\sup \langle b_n\rangle$ in $\W(A\otimes\mathcal{K})$.
\end{proof}

It is well known that if a simple unital C$^*$-algebra $A$ has strict comparison if and only if $\W(A)$ is \emph{almost unperforated}, that is, whenever $(n+1)x\leq ny$ ($x$, $y\in \W(A)$, $n\in \mathbb{N}$), one has $x\leq y$ (see \cite{rorfunct}). A related property that a partially ordered monoid $M$ might satisfy is that of being \emph{almost divisible}. This means that, for any $x$ in $M$ and any $n\in\mathbb{N}$, there is $y$ in $M$ such that $ny\leq x\leq (n+1)y$.

Recall that a C$^*$-algebra $A$ is $\mathcal{Z}$-stable if it absorbs $\mathcal{Z}$ tensorially, that is, $A\otimes\mathcal{Z}\cong A$, where $\mathcal{Z}$ is the Jiang-Su algebra (\cite{jiangsu}). Any $\mathcal{Z}$-stable algebra $A$ has the properties that $\W(A)$ is almost unperforated and almost divisible (see \cite{rorijm}, \cite{bpt}, \cite{kngper}). This is also the case for every simple, unital AH-algebra with slow dimension growth (\cite{bpt}, \cite{tomsplms}). 

For the class of simple, separable, finite, exact, unital C$^*$-algebras with strict comparison and such that $\W(A)$ is almost divisible, a precise description of $\W(A)$ was given in \cite{bpt} (see also \cite{aptsantander}), by means of data already contained in the Elliott invariat (i.e. $\mathrm{K}$-Theory and traces). Moreover, the recovery of one from the other has a functorial nature (see \cite{pt}).

More specifically, consider the set $\V(A)\sqcup \mathrm{LAff}(\T(A))^{++}$. Equip it with an abelian monoid structure and a partial order, as follows. Addition extends the natural operations in both $\V(A)$ and $\mathrm{LAff}(\T(A))^{++}$, and is defined on mixed terms as $[p]+f=\hat{p}+f$ where $\hat{p}(\tau)=\tau(p)$.
The order is given by:
\begin{enumerate}[{\rm(i)}]
\item $\leq$ is compatible with the natural order defined in $\V(A)$;
\item if $f,g\in\mathrm{LAff}(\T(A))^{++}$ we will say that
$f\leq g$ if $f(\tau)\leq g(\tau)$ for all
$\tau\in \T(A)$; 
\item If $f\in\mathrm{LAff}(\T(A))^{++}$ and
$[p]\in \V(A)$ we will say that $f\leq [p]$ if $f(\tau)\leq \tau(p)$ for all $\tau\in \T(A)$;
\item If $f\in\mathrm{LAff}(\T(A))^{++}$
and $[p]\in \V(A)$ we will say $[p]\leq f$ if $\tau(p)<f(\tau)$ for all $\tau\in \T(A)$.
\end{enumerate}

The set $\V(A)\sqcup\mathrm{LAff}_b(\T(A))^{++}$ naturally inherits the structure and order just defined, so that the natural inclusion $\iota\colon \V(A)\sqcup\mathrm{LAff}_b(\T(A))^{++}\to \V(A)\sqcup\mathrm{LAff}(\T(A))^{++}$ is an order-embedding.

\begin{theorem}
\label{thm:bpt} {\rm (cf. \cite{bpt}, \cite{aptsantander})}
Let $A$ be a simple, unital, separable, exact C$^*$-algebra. Assume that $A$ is finite, has strict comparison and $\W(A)$ is almost divisible. Then, there is an ordered monoid isomorphism
\[
\W(A)\cong \V(A)\sqcup \mathrm{LAff}_{b}(\T(A))^{++}\,,
\]
such that $\langle p\rangle\mapsto [p]$, for a projection $p$, and $\langle a\rangle\mapsto\hat{a}$, for $a$ not equivalent to a projection.
\end{theorem}

The following lemma is known. Its proof uses Edwards' separation Theorem (see, e.g. \cite[Theorem 11.12]{poag}).
\begin{lemma}\label{LAff2}
Let $K$ be a metrizable Choquet simplex and $g\in\mathrm{LAff}(K)^{++}$. Then $g$ is the pointwise supremum of a strictly increasing sequence $(f_{n})$, where $f_{n}\in \mathrm{Aff}(K)^{++}$ for all $n$.
\end{lemma}

Let $f$ and $g$ be affine functions on  a convex set $K$. We write $f<g$ to mean 
$f(x)<g(x)$ for every $x$ in $K$.

\begin{lemma}\label{LAff3}
Let $K$ be a compact convex set, and let $f$, $g\in\mathrm{LAff}(K)^{++}$. 
\begin{enumerate}[\rm(i)]
\item If $f\ll g$, then $f< g$.
\item If $f<g$ and $f$ is continuous, then $f\ll g$.
\end{enumerate}
\end{lemma}

\begin{proof}
(i). Since $g\in\mathrm{LAff}(K)^{++}$ and $K$ is compact, we know that $g$ is bounded away from zero. 
Therefore, there exists $n_{0}$ such that $g-1/n\in\mathrm{LAff}(K)^{++}$ for all $n\geq n_{0}$, and we may take $n_{0}=1$. Since $\sup_{n}(g-1/n)=g$, there exists $n$ such that 
$f\leq g-1/n$, whence $f<g$.

(ii). Suppose that $g\leq\sup(g_{n})$ where $(g_{n})$ is an increasing sequence in $\mathrm{LAff}(K)^{++}$.
For each $n$, put $\mathcal{U}_{n}:=\{x\in K\mid f(x)< g_{n}(x)\}$, which is open as $f$ is continuous and $g_{n}$ is lower semicontinuous (and so is $g_{n}-f$). Since $f<g$, we see that $\bigcup_{n\geq 1}\mathcal{U}_{n}=K$. Using now that $K$ is compact and that $(g_n)$ is increasing, we find $m\geq 1$ with $K=\mathcal{U}_{m}$. This implies that $f< g_{m}$.
\end{proof}

\begin{lemma}
\label{lem:whatsthat}
Let $A$ be a simple, separable and unital C$^*$-algebra. Then, the monoid $\V(A)\sqcup \mathrm{LAff}(\T(A))^{++}$ is an object of Cu.
\end{lemma}
\begin{proof}
Let $M=\V(A)\sqcup \mathrm{LAff}(\T(A))^{++}$. Let us first prove that every increasing sequence $(x_n)$ in $M$ has a supremum. If infinitely many terms of the sequence belong to $\mathrm{LAff}(\T(A))^{++}$, then the supremum equals the (pointwise) supremum $g$ of the functions that appear in $(x_n)$. For, if $[p]\in\V(A)$ is such that $x_n\leq [p]$ for all $n$, then by definition $x_n\leq\hat{p}$ for every $n$ such that $x_n\in\mathrm{LAff}(\T(A))^{++}$, whence $g\leq \hat{p}$, that is, $g\leq [p]$.

Otherwise, all but finitely many terms in $(x_n)$ belong to $\V(A)$. For those, write $x_n=[p_n]$, where $p_n$ are projections in matrices. Then, either the sequence is eventually constant, in which case the supremum belongs to $\V(A)$, or else $[p_k]<[p_{k+1}]$ for infinitely many $k$s. In that case $\sup_n x_n=\sup_n \hat{p}_n$. We only need to verify that $[p_k]\leq \sup_n\hat{p}_n$. Indeed, given $k$, there is $l>k$ with $[p_k]<[p_l]$. Simplicity of $A$ ensures that $\hat{p}_k<\hat{p}_l$, whence $\hat{p}_k<\sup_n\hat{p}_n$. Thus $[p_k]\leq\sup_n\hat{p}_n$.

From our observations above, it follows that the (only) compact elements in $M$ are the ones in $\V(A)$. Indeed, if $[p]\leq\sup f_n$ for functions $f_n\in \mathrm{LAff}(\T(A))^{++}$, then $\hat{p}<\sup f_n$, and by compactness we may choose $\epsilon>0$ such that $\hat{p}+\epsilon<\sup f_n$. Since $\hat{p}$ is continuous, Lemma \ref{LAff3} implies $\hat{p}<f_k$ for some $k$, that is, $[p]\leq f_k$.

As $A$ is separable, we know that $\T(A)$ is metrizable, and it follows from Lemma \ref{LAff2} that each element in $\mathrm{LAff}(\T(A))^{++}$ is the pointwise supremum of a strictly increasing sequence of elements from $\mathrm{Aff}(\T(A))^{++}$, which again, in view of Lemma \ref{LAff3}, is a rapidly increasing sequence.

It is easy to verify that suprema and addition in $M$ are compatible. Assume now that $x\ll y$, and $z\ll t$. The only case of interest arises when one of $y$, $t$ belongs to $\V(A)$ and the other does not. Assume then that, for example, $y\in\V(A)$, and write $t=\sup_n f_n$ for a strictly increasing sequence of affine, continuous functions. Then, $z\leq f_n$ for some $n$, whence
\[
x+z\leq y+t_n=\hat{y}+f_n\ll\hat{y}+f_{n+1}=y+f_{n+1}\leq y+t\,,
\]
showing that $\ll$ and addition are compatible.
\end{proof}

Assembling the results above, we obtain the following result that recovers Theorem 2.6 in \cite{bt}.
\begin{theorem}
Let $A$ be a unital, simple, separable, exact C$^*$-algebra. Assume that $A$ is finite, has strict comparison, and $\W(A)$ is almost divisible. Then, there is an ordered monoid isomorphism
\[
\W(A\otimes\mathcal{K})\cong \V(A)\sqcup\mathrm{LAff}(\T(A))^{++}\,.
\]
\end{theorem}
\begin{proof}
Let $\iota\colon\V(A)\sqcup \mathrm{LAff}_{b}(\T(A))^{++}\to \V(A)\sqcup \mathrm{LAff}(\T(A))^{++}$ be the natural inclusion. We only need show that the pair $(\V(A)\sqcup\mathrm{LAff}(\T(A))^{++},\iota)$ is the completion of $\V(A)\sqcup \mathrm{LAff}_{b}(\T(A))^{++}$ and then invoke Theorem \ref{thm:bpt} and Theorem \ref{thm:wacompletion}.

We know that $\V(A)\sqcup \mathrm{LAff}_{b}(\T(A))^{++}$ is an object of PreCu (by Theorem \ref{thm:bpt}, Proposition \ref{prop:wainprecu} and Lemma \ref{lem:sr1her}), and that $\V(A)\sqcup \mathrm{LAff}(\T(A))^{++}$ is an object of Cu (by Lemma \ref{lem:whatsthat}). Clearly, the map $\iota$ is an order-embedding in PreCu. Now, any $f\in \mathrm{LAff}(T(A))^{++}$ may be written as $f=\sup(\iota(f_{n}))$ where $f_{n}\in\mathrm{Aff}(T(A))^{++}$. Thus $(\V(A)\sqcup \mathrm{LAff}(\T(A))^{++},\iota)$ satisfies the requirements of Definition \ref{dfn:completion}.
\end{proof}

We now turn to the real rank zero situation, and show that, for this class, the Cuntz semigroup (of the stabilisation) is order isomorphic to the monoid of intervals in the projection monoid. This was shown to be the case if moreover the algebra has stable rank one by the third author in \cite{perijm}, but this turns out not to be necessary.

\begin{theorem}
Let $A$ be a $\sigma$-unital C$^*$-algebra with real rank zero.
Then
\[
\W(A\otimes\mathcal{K})\cong\Lambda_{\sigma}(\V(A))
\]
as ordered monoids.
\end{theorem}

\begin{proof}
We may identify $\V(A\otimes\mathcal{K})$ with $\V(A)$. For each $a\in A\otimes\mathcal{K}_+$ put $I(a)=\{[p]\in\V(A)\mid p\precsim a\}=\{[p]\in\V(A)\mid p\in\overline{aM_{\infty}(A)a}\}$, which is an element in $\Lambda_{\sigma}(\V(A))$.

Define $\varphi\colon\W(A\otimes\mathcal{K})\to\Lambda_{\sigma}(\V(A))$
by $\varphi(\langle a\rangle)= I(a)$. The proof of Theorem 2.8 in \cite{perijm} shows that $\varphi$ is a well defined order embedding.

Now, let $I\in\Lambda_{\sigma}(\V(A))$ and let $\{[p_n]\}$ be an increasing
countable cofinal subset. We of course have that $(\langle p_n\rangle)$ is an increasing sequence in $\W(A\otimes\mathcal{K})$. Let $\langle a\rangle=\sup(\langle p_{n}\rangle)$, and let us show that $I(a)=I$.

Let $p\precsim a$ be a projection. Then $\langle
p\rangle\leq\langle a\rangle$. Since $\langle
p\rangle\ll\langle p\rangle\leq\langle a\rangle =\sup(\langle
p_{n}\rangle)$, we get $\langle p\rangle\leq\langle p_{n}\rangle$
for some $n$, that is, $[p]\leq [p_{n}]$ in $\V(A)$. Conversely, if $[q]\in I$, then
$[q]\leq [p_{n}]$ for some $n$, and therefore $\langle q\rangle \leq
\langle p_{n}\rangle\leq \langle a\rangle$. Thus $[q]\in I(a)$.
\end{proof}

\section{Countable inductive limits of monoids}\label{sec:limits}

In this section we answer, in the negative, the question of whether the category PreCu is closed under coutable inductive limits. We subsequently repair this defect by constructing a smaller category that sits between Cu and PreCu, and to which the Cuntz semigroup belongs in interesting cases.

\begin{theorem}\label{PreCunolim}The category PreCu does not have inductive limits.\end{theorem}
\begin{proof}
We consider for all $i\geq 0$ the following monoids
\[S_i=\frac{1}{2^i}\N=\left\{\frac{n}{2^i}\ \mid n\in \N\right\},\]
with the natural order and addition. It is clear that these discrete monoids are in PreCu (in fact in $\Cu$ if we add an infinite element) and are all isomorphic to $\N$. Consider now the following inductive sequence, where $f_i$ are the natural inclusions:

\[ S_1 \stackrel{f_1}{\to} S_2  \stackrel{f_2}{\to} S_3  \stackrel{f_3}{\to} \dots\]

Let $S=\bigcup_{i\geq 1} S_i\subseteq \Q^+$ be the standard algebraic inductive limit of the sequence
(i.e. the inductive limit in the category of ordered abelian monoids), with the inclusions $\varphi_i\!:\!S_i\to S$ as compatible maps 
($\varphi_{i+1}f_i=\varphi_i$). Observe that this monoid is no longer discrete, and, as we will see, it can not be the limit in 
PreCu (neither in $\Cu$ by just adding infinity), since properties as $1\ll 1$ in any of the $S_i$ are not preserved in $S$.

Arguing by contradiction, suppose that the sequence $(S_i, f_i)$ has an inductive limit in PreCu, 
\[ T=\displaystyle{\lim_{\rightarrow\text{\tiny PreCu}}}\left(S_i,f_i\right) \text{ with } \psi_i\!:\!S_i\to T \text{ such that }\psi_{i+1}f_i=\psi_i.\] 

We will construct two monoids $T_1$, $T_2$ in PreCu with compatible maps $S_i\to T_j$ ($j=1,2$), and use the universal property from $T$ to derive a contradiction.

First, let $S'=S\setminus\{0\}$. Denote its elements by $s'$. Define $T_1:=S\sqcup S'$ with addition extending the addition in $S$ and $S'$ by $(a+b')=(a+b)'$ (primes denoting elements in $S'$),
and the order also extending the order in $S$ and $S'$ with 
\[ a'\leq b \text{ if } a\leq b,\ \text{ and }  a\leq b' \text{ if } a<b.\] 
This last condition is in fact a consequence of extending algebraically the order in $S$ and $S'$.
We also define $T_2:=S\sqcup S'$, the same as $T_1$ as sets and even as monoids, but with a different order, again extending the order from  $S$ and $S'$ but now 
\[ a'\leq b \text{ if } a< b, \ \text{ and } a\leq b' \text{ if } a<b. \]
Observe that, by doing this, we are just preventing $a$ and $a'$ to be comparable for any $a\in S\setminus\{0\}$.

\smallskip
\emph{\underline{Claim 1:} $T_1$ and $T_2$ are objects of PreCu.}
\smallskip

It is easy to see, in both situations, that the addition is order preserving and hence that $T_1,T_2$ are ordered abelian monoids.
The other properties in PreCu are also easily derived once we know how suprema are constructed, and, with that, when we have 
$x\ll y$. 

Suprema of stationary sequences always exist, so let us consider $x_1\leq x_2\leq \dots \in T_j$ a non stationary sequence with 
a supremum in $T_j$. Let $\gamma_j\!:\!T_j\to S$ be the order preserving monoid map identifying both copies of $S$, and consider 
the sequence $(\gamma_j(x_i))_i$ which is also non stationary since we only have two copies of each element. Observe that 
if $\gamma_j(x_i)<r$ in $S$, then $x_i<r,r'$ both in $T_1$ and $T_2$. Hence, if a supremum exists in $T_j$, it must exist in
$S$, say $r=\sup_{S}(\gamma_j(x_i))_i$, and the supremum in $T_j$ must be either $r$ or $r'$.
In $T_1$ we have $r'<r$ therefore $\sup_{T_1}(x_n)_n=r'$. However, $r$ and $r'$ are not comparable in $T_2$, and hence we can not have a supremum.

Hence, in $T_1$, non stationary sequences  with suprema are the ones which have a supremum by identifying both copies of $S$, and in this case, the supremum is the copy coming from $S'$. In $T_2$ the only sequences with supremum are the stationary ones.

Using this one sees that $T_1$ is a linearly ordered monoid, with two copies of each element in $S\setminus \{0\}$, ordered as $a'<a$, 
and with $x\ll y$ if $x<y$ or if $x=y\in S$. 

Suprema in $T_2$ come only from stationary sequences, whence this is like the discrete situation, so we have 
$x\ll y$ if $x\leq y$.

With this in mind it is now easy to prove that $T_1$ and $T_2$ are objects in PreCu.
Furthermore, considering the natural inclusions in the first copy of $S$, $i_j\colon S\to T_j$ for $j=1,2$, we also obtain compatible maps
in PreCu, $i_j\varphi_i\colon S_i\to T_j$ (that is, $i_j\varphi_i=i_j\varphi_{i+1}f_i$).

\smallskip
\emph{\underline{Claim 2:} $T$ contains a copy of $S$ as an ordered abelian submonoid.}
\smallskip

Consider the following diagram of ordered monoid maps, observing that the only maps that belong to PreCu are the $f_i$s and the $i_1\varphi_j$s:
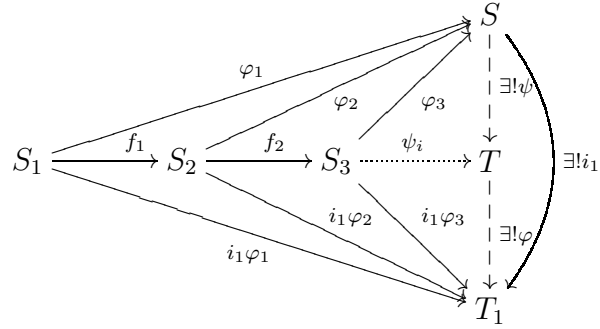
\[
\xymatrix@=7ex{ 
                                     &                                                                       &           & S\ar@{-->}[d]^{\exists !\psi}\ar@/^2pc/[dd]^{\exists ! i_1} \\
S_1\ar@{->}[r]^(.7){f_1} \ar@{->}[urrr]^{\varphi_1} \ar@{->}[drrr]_{i_1\varphi_1} & S_2\ar@{->}[r]^(.6){f_2} \ar@{->}[urr]_{\varphi_2}\ar@{->}[drr]^{i_1\varphi_2}  & S_3\ar@{..>}[r]^{\psi_i} \ar@{->}[ur]_{\varphi_3} \ar@{->}[dr]^{i_1\varphi_3}  & T \ar@{-->}[d]^{\exists !\varphi} \\ 
                                    &                                                                       &           & T_1
}
\]

Since $T$ has compatible maps $\psi_i$ with the $f_i$,  by the universal property of $S$ we obtain a unique ordered monoid homomorphism $\psi\!:\!S\to T$ such that $\psi\varphi_i=\psi_i$. Again, since $T_1$ also has compatible maps $i_1\varphi_i$, now in the category PreCu, by the universal property of $T$ we obtain a unique map $\varphi\!:\!T\to T_1$ such that $\varphi\psi_i=i_1\varphi_i$, and,
at the level of monoids, we obtain a unique ordered map $\tilde\varphi\!:\!S\to T_1$ such that $\tilde\varphi\varphi_i=i_1\varphi_i$. By uniqueness this last map should be the inclusion, $i_1=\tilde\varphi$, and since $\varphi\psi\varphi_i=\varphi\psi_i=i_1\varphi_i$ we have, again by uniqueness, $i_1=\varphi\psi$.

Therefore, since the inclusion factors through $T$, $\psi$ is injective and we have an isomorphic copy of $S$ in $T$. We will thus write $T=\psi(S)\sqcup T'$.

\smallskip
\emph{\underline{Claim 3:} $T'\neq\emptyset$.}
\smallskip

Observe that since the monoids $S_i$ are discrete and all bounded sequences are eventually stationary,
all elements are compact (that is, $a\ll a$ for all $a\in S_i$). 
Now for any $i\geq 1$, $\psi(1)=\psi(\varphi_i(1))=\psi_i(1)$, and since $1\ll 1$ in $S_i$, and $\psi_i$ are PreCu morphisms,
we have that $\psi(1)\ll \psi(1)$ also in  $T$. Clearly  this can be extended to prove that all elements in $\psi(S)$ are compact.

In $T$ we have that $\psi(1)\geq \psi(\frac{2^n-1}{2^n})=\psi(1-\frac{1}{2^n})$ for all $n\geq 1$. Since $\psi$ is 
injective and compact elements can not be written as suprema of non stationary sequences, we have that 
$\psi(1)\neq \sup_T(\psi(1-\frac{1}{2^n}))_n$ (this last element might not even exist). But $\psi(1)$ is an upper bound for the sequence
and therefore there exists $t'\in T$ such that $\psi(1-\frac{1}{2^{n}})\leq t'$, for all $n\geq 1$, and
either $t'<\psi(1)$ or else $t'$ is not comparable with $\psi(1)$. Since $\psi$ is an ordered morphism, 
and $\psi(S)$ is completely ordered, it is clear that $t'\not\in \psi(S)$, and hence $t'\in T'\neq \emptyset$.

\smallskip
\emph{\underline{Claim 4:} $T$ cannot be the inductive limit in PreCu of $(S_i,f_i)$.}
\smallskip

Recall that in $T_2$ the only sequences with suprema are the stationary sequences. With this in mind, and with $\psi$ as in 
Claim 2, it is not difficult to see that the following is a morphism in PreCu :
\[
\begin{array}{rcl} \gamma: T_2& \longrightarrow &T \\
a & \longmapsto & \psi(a) \text{ if } a\in S \\
a' & \longmapsto & \psi(a) \text{ if } a'\in S'
\end{array}
\]
Now consider the following diagram of maps in PreCu, 
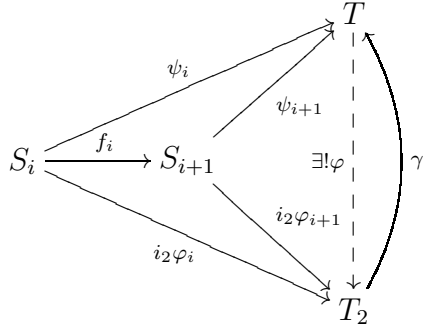
\[ \xymatrix@=7ex{
& & T \ar@{-->}[dd]_{\exists ! \varphi} \\
S_i\ar[r]^{f_i}\ar[urr]^{\psi_i}\ar[drr]_{i_2\varphi_i} & S_{i+1}\ar[ur]_{\psi_{i+1}}\ar[dr]^{i_2\varphi_{i+1}} & \\
& & T_2\ar@/_1.5pc/[uu]_\gamma} 
\]
By the universal property of $T$ there exists a unique map $\varphi\!:\!T\to T_2$ such that $\varphi\psi_i=i_2\varphi_i$. 
Now observe that $\psi_i(s)=\psi(\varphi_i(s))=\gamma(i_2 \varphi_i(s))$ and therefore, $\gamma\varphi\psi_i=\gamma i_2\varphi_i=\psi_i$.
But $\gamma\varphi\neq \text{Id}_T$ since $t'\not\in \gamma(T_2)$. 
This contradicts the universal property for $T$ leading to two different compatible maps from $T$ to $T$.
\end{proof}

\begin{remark}
{\rm Observe that the example of inductive chain used in the proof of the previous Theorem, can be obtained as an inductive chain induced by 
the Cuntz functor $\W(-)$ applied to an inductive chain of C$^*$-algebras. 

Consider the $2^\infty$ UHF-Algebra, $A=\lim_{\rightarrow} (M_{2^i}(\C),g_i)$ with $g_i(x)=\left(\begin{smallmatrix}x & 0 \\ 0 & x\end{smallmatrix}\right)$. The Cuntz semigroup of each matrix algebra $\W(M_n(\C))$ is isomorphic to a cyclic free semigroup with $0$, given by the rank function, $\langle a\rangle \mapsto \text{rank}(a)$ for all $a\in A_+$. Hence, if we make identifications $\W(M_{2^i}(\C))=\frac{1}{2^i}\N=S_i$ -- by using the weighted rank funtions $\langle a\rangle\mapsto \frac{1}{2^i}\text{rank}(a)$ -- the induced maps $\W(g_i)$ are the natural inclusions $f_i\colon S_i\to S_{i+1}$. Therefore, we obtain the previous sequence of monoids in PreCu.}
\end{remark}

Let us thus define a new category, which will be suitable for the Cuntz semigroup for a large class of C$^*$-algebras, and in which inductive limits can be constructed.

\begin{definition}
Let $\mathcal{C}$ be the full subcategory of PreCu whose objects are monoids $M$ closed by suprema of bounded increasing sequences.
\end{definition}

\begin{examples}{\rm 
As illustrating examples, observe that  $\mathbb{Q}^+$ is an object of  PreCu but not of ${\mathcal C}$, $\mathbb{R}^+$ is an object of ${\mathcal C}$ but not of Cu, and finally $\mathbb{R}^+\cup \{\infty\}$ is an object of $\Cu$.

We have seen in Lemma~\ref{lem:sr1her} that for  stable rank one C$^*$-algebras $A$, $\W(A)$ is hereditary and
therefore in PreCu. Also, by~\cite[Lemma 4.3]{bpt} we see that $\W(A)$ has suprema of bounded sequences (this was in fact used to prove Lemma~\ref{lem:sr1her}) and
therefore  $\W(A)$ is an object of $\mathcal C$. 

Furthermore, we will see that the category $\mathcal C$ coincides with the category of monoids $M$ for which 
the inclusion $\iota\colon M\to \overline{M}$ is hereditary.}.
\end{examples}

\begin{proposition}\label{prop:heriffC} Let $M$ be in PreCu. Then the inclusion $\iota\colon M\to\overline{M}$ is hereditary if and only if $M$ is an object of $\mathcal C$, that is, all bounded
increasing sequences in $M$ have a supremum. \end{proposition}

\begin{proof} Suppose the inclusion $\iota\colon M\to \overline{M}$ is hereditary and consider $(x_n)$ a bounded ascending sequence in $M$, say $x_n\leq y\in M$ for
all $n\geq 1$.  Then, since $\iota$ is a map in PreCu, we obtain $\iota(x_n)\leq \iota (y)$ for all $n\geq 1$. Now, since $\overline{M}$ is in Cu, we have  $z=\sup_{\overline{M}}(\iota(x_n))\leq \iota(y)$. Using that $\iota$ is
hereditary there exists an element $x\in M$ such that $\iota(x) = z$. Suppose there exists $x'\in M$ such that $x_n\leq x'$ for all $n$. Then $\iota(x_n)\leq \iota(x')$ which implies $\iota(x)=z\leq \iota(x')$. But since $\iota$ is an order-embedding we have $x\leq x'$ and therefore 
$x=\sup_M (x_n)$.

Now suppose $M$ is a monoid in $\mathcal C$ and consider the map $\iota \colon M\to \overline{M}$. 
Since $(\overline{M},\iota)$ is a completion of $M$, any $x\in \overline{M}$  can be written as $x=\sup_{\overline{M}}(\iota(x_n))$. Suppose further that 
$x\leq \iota(y)$ for some $y\in M$. Therefore $\iota(x_n)\leq \iota(y)$ for all $n\geq 1$, and since $\iota$ is an order embedding,
we obtain $x_n\leq y$ for all $n\geq 1$. Now since $M$  is in $\mathcal C$ and the sequence is
bounded by $y$, there exists $z=\sup_{M}(x_n)$ and since $\iota$ preserves suprema we obtain $\iota(z)=\iota(\sup_M(x_n))=\sup_{\overline{M}}(\iota(x_n)) = x$, obtaining the desired result.
\end{proof}

\begin{remark}
Note that for a C$^*$-algebra $A$, if the embedding $\W(A)\to \W(A\otimes\mathcal{K})$ is hereditary, by Theorem~\ref{thm:wacompletion} $\W(A)$ is an object of PreCu and $\W(A\otimes\mathcal{K})$ is order-isomorphic to $\overline{W(A)}$. Therefore the embedding  $\W(A)\to \overline{\W(A)}$ is also hereditary.
On the other hand, it is not clear that if $\W(A)$ is an object of PreCu and the embedding $\W(A)\to\overline{\W(A)}$ is hereditary then $\W(A)\to\W(A\otimes\mathcal{K})$ is also hereditary
\end{remark}

By Proposition~\ref{prop:heriffC} and the preceding remark, we obtain:

\begin{corollary} If $A$ is a C$^*$-algebra such that $\W(A)$ is hereditary, then $\W(A)$ is an object of $\mathcal C$. 
\end{corollary}

\begin{theorem}
The category $\mathcal C$ has countable inductive limits.
\end{theorem}

\begin{proof} 

Let $(S_i,f_i)_{i\geq 0}$ be an inductive sequence of monoids in $\mathcal C$ and let $S^{\text{alg}}$ be the algebraic inductive limit with compatible maps $\varphi_i:S_i\to S^{\text{alg}}$. Also, we define for any $m>i\geq 0$ the maps $f_{m,i}=f_{m-1}\dots f_i$.

To construct the inductive limit in $\Cu$, Coward, Elliott and Ivanescu considered the set of ascending sequences (through 
the morphisms $f_i\!:\!S_i\to S_{i+1}$) with an intertwining relation between the compactly contained elements of the sequence. 
We will use the same construction but, since we are only interested in obtaining suprema in $S$ for bounded sequences, we should only be interested in sequences which are bounded in some of the $S_i$, and its successive homomorphic images, that is, in $S^{\text{alg}}$. But, in order to maintain
the rapidly increasing structure, we will consider ascending sequences with a  bound in $S^{\text{alg}}$ for its compactly contained elements.

Let us call a sequence $s=(s_1,s_2,\dots )$ with $s_i\in S_i$ a \emph{bounded ascending sequence} in $(S_i,f_i)$ if 
$f_i(s_i)\leq s_{i+1}$ and there exists $M_s\in S^{\text{alg}}$ such that, for all $i\geq 0$  and $x\ll s_i$, $\varphi_i(x)\leq M_s$. We will say that $M_s$ is the \emph{bound in} $S^{\text{alg}}$ for $s$, or that $x$ is bounded in $S^{\text{alg}}$.
We now define 
\[ S^{(0)}:=\{ s=(s_1,s_2,\dots )  \mid s \text{ is bounded}\}.\]
This set becomes a pre-ordered abelian monoid with  componentwise addition and pre-order relation given by $(s_1,s_2,\dots)\precsim (t_1,t_2,\dots)$ if, for any $i$ and $s\in S_i$ with $s\ll s_i$, there exists an $m>i$ such that $f_{m,i}(s)\ll t_m$. 
Antisymmetrizing the relation $\precsim$ ($(s_i)\sim (t_i)$ if $(s_i)\precsim (t_i)$ and $(t_i)\precsim (s_i)$), we obtain an ordered abelian monoid  $S=S^{(0)}/\sim$ which, together with the morphisms $\varphi_i\!:\!S_i\to S$, $\varphi_i(s)=[(0,\dots,0,s,f_i(s),\dots)]$,
is the inductive limit of $(S_i,f_i)$ in $\mathcal C$.

This construction is based on the construction in \cite{CEI} with the difference that we are considering 
a wider range of monoids (possibly with unbounded sequences), but a smaller subset $S^{(0)}$ (subset of the cartesian product $\prod_{i}S_i$). The proof follows the lines of the one in \cite{CEI} but with a number of non trivial modifications. 
For the sake of brevity we will point out where extra care in the construction needs to be taken.

First, we need to check that $S^{(0)}$ is closed under addition.

Let $(s_1,s_2,\dots), (t_1,t_2,\dots)\in S^{(0)}$ with bounds in $S^{alg}$ $M_s,M_t$ respectively. For any $i\geq 0$ and $x\ll s_i+t_i$, let us write $s_i,t_i$ as suprema of rapidly increasing sequences in $S_i$, $s_i=\sup_{S_i}(s_i^n)_n$, and $t_i=\sup_{S_i}(t_i^n)_n$. Then $x\ll s_i+t_i=\sup_{S_i}(s_i^n+t_i^n)_n$. Hence, for some $m$ 
we have that $x\ll s_i^m+t_i^m$. But $s_i^m\ll s_i$ and $t_i^m\ll t_i$ implies that $\varphi_i(s_i^m)\leq M_s$ and $\varphi_i(t_i^m)\leq M_t$. Therefore $\varphi_i(x)\leq \varphi_i(s_i^m+t_i^m)\leq M_s+M_t\in S^{\text{alg}}$ which is a bound in $S^{\text{alg}}$ for $(s_i)+(t_i)$ . 

One other important fact is that each element $[(s_i)]\in S$ has a representative 
$[(\tilde s_i)]$ whose elements are rapidly increasing in the sense that $f_i(\tilde s_i)\ll \tilde s_{i+1}$. This can still be done, since 
this representative is a Cantor diagonal sequence obtained from rapidly increasing sequences $s_i^n$ with the original elements $s_i$ as suprema. Since those elements are always way below the original ones, the resulting representative is also bounded in $S^{\text{alg}}$. 

Also, we should take care in how suprema is constructed. 

First we see that if $(s_i)\precsim (t_i)$ and $M_t\in S^{\text{alg}}$ is a bound for $(t_i)$, then $M_t$ is also a bound for $(s_i)$.
If $x\ll s_i$, by the $\precsim$ relation, there exists $m$ such that $f_{m,i}(x)\ll t_m$. Now, by the bound for the compactly contained in $t_m$, $\varphi_m(f_{m,i}(x))\leq M_t$. But since the $f_j$ are compatible maps, we obtain $\varphi_i(x)\leq M_t$.

Therefore, a bound in $S$ for an ascending sequence $s^i$ gives us a bound in $S^{alg}$  for all the elements in the sequence, and thus 
for the computed supremum: Let $[(s_i^1)]\leq [(s_i^2)]\leq\dots \leq [(t_i)]$ be a bounded ascending sequence in $S$. By the above argument the compactly contained elements of all the $s_i^j$ (even for any other representative), are bounded in $S^{\text{alg}}$
by the bound of the $(t_i)$, say $M_t$. 
But the supremum (as constructed in \cite{CEI}) is obtained from the components of the $[(s_i^n)]$ (rapidly increasing representatives),
which will be bounded in $S^{\text{alg}}$ by $M_t$, and therefore will also led to an element in $S^{(0)}$.
\end{proof}

\begin{remark}
{\rm Let us recall a property of the inductive limit in $\Cu$, which is still valid in $\mathcal C$  and which should be used later. As in the proof of the previous Theorem, given $s\in S$, choosing a rapidly increasing representative, $s=[(s_i)]\in S$ with $f_i(s_i)\ll s_{i+1}$,
we have that $s=\sup_{S}(\varphi_i(s_i))_i$,
that is, all elements in $S$ can be written as the supremum of a sequence of elements coming from the $S_i$.}
\end{remark}

\begin{remark}
{\rm Observe that the inductive sequence in Theorem~\ref{PreCunolim} is in fact a sequence in $\mathcal C$.
To compute the inductive limit $\lim_{\rightarrow\mathcal C}(S_i,f_i)$, just observe that $T=S\sqcup \mathbb{R}^{++}$ 
with order and addition as $T_1$ in the proof of Theorem~\ref{PreCunolim}, has the desired properties.

If, instead of $(S_i,f_i)$  we had considered the sequence $(S_i\cup\{\infty\},\bar f_i)$ (now objects in $\Cu$),
we would have obtained $\lim_{\rightarrow\mathcal C}(S_i,f_i)=S\sqcup \R^{++}\cup\{\infty\}$ which is the inductive limit in $\Cu$.

One could ask whether or not, the limit in $\mathcal C$  applied to monoids already in $\Cu$ leads to the inductive limit  in $\Cu$, that is,
if $S_i$ are monoids in $\Cu$, \[\lim_{\rightarrow \mathcal C}S_i\stackrel{?}{=}\lim_{\rightarrow \Cu}S_i.\]
This is not the case as we see in the following example:

For all $n\geq 0$ let $T_n=\{a_0,a_1,\dots,a_n\}$ with addition $a_i+a_j=a_{\text{max\{i,j\}}}$ and the induced algebraic order. 
It is not difficult to check that $T_n\in\Cu$, and now, if we consider the inductive sequence $(T_n,g_n)$ with the natural inclusions as maps,
it can be proven, through its defining universal properties, that 
\[  \lim_{\rightarrow \mathcal C}S_i=\{a_0,a_1,a_2,\dots\},\text{ and }  \lim_{\rightarrow \Cu}S_i=\{a_0,a_1,a_2,\dots,a_{\infty}\}.\]
equipped with the same order and addition as before.}
\end{remark}

\begin{theorem}\label{thm:limitsac}
Let $(S_i,f_i)$  be an inductive sequence of maps in $\mathcal C$. Then 
\[ \overline{ \lim_{\rightarrow \mathcal C} S_i}= \lim_{\rightarrow{\mathcal C}u} \overline{S_i}.\]
\end{theorem}

\begin{proof}
Let $(S_i,f_i)$ be an inductive sequence of maps in $\mathcal C$, and $S_{\mathcal C}$ the inductive limit in $\mathcal C$ with compatible maps $\varphi_i\!:\!S_i\to S_{\mathcal C}$.
Similarly let $(\overline{S_i}, \overline{f_i})$ be the induced sequence in $\Cu$, and $S_{\Cu}$ the inductive limit in $\Cu$ with
compatible maps $\psi_i\!:\overline S_i\to S_{\Cu}$. Let $\gamma_i\!:\! S_i\to \overline S_i$ be the corresponding inclusions.
Now, since $\gamma_{i+1}f_i=\overline{f_i}\gamma_i$ (by definition as in Theorem~\ref{universality}) consider the following diagram of compatible maps in $\mathcal C$:
\[
\xymatrix{
& & S_{\mathcal C}\ar@{-->}[ddd]^{\exists!\gamma} \\
S_i\ar[r]^{f_i}\ar@/^/[urr]^{\varphi_i}\ar[d]_{\gamma_i} & S_{i+1}\ar[ur]_{\varphi_{i+1}}\ar[d]^{\gamma_{i+1}} & \\
\overline{S_i}\ar[r]^{\overline{f_i}}\ar@/_/[drr]_{\psi_i} & \overline{S_{i+1}}\ar[dr]^{\psi_{i+1}} & \\
& & S_{\Cu}
}
\]
By the universal property of $S_{\mathcal C}$ there exists a unique map $\gamma\!:\!S_{\mathcal C}\to S_{\Cu}$ such that $\gamma\varphi_i=\psi_i\gamma_i$.

If $s\in S_{\Cu}$, then $s$ can be written as a supremum of a rapidly increasing sequence of elements coming from the $\overline S_i$ (see \cite{CEI}), 
$s=\sup_{S_{\Cu}}(\psi_i(s_i))_i$. In turn, since $\overline{S_i}$ is the completion of $S_i$, each of the $s_i$ can be written as a supremum of a rapidly increasing sequence of elements coming from $S_i$, $s_i=\sup_{\overline S_i}(\gamma_i(s_i^j))$. Therefore, since $\psi_i$ are morphisms in $\Cu$ preserving suprema,
\begin{multline*}s=\sup_{S_{\Cu}}(\psi_i(s_i))_i=\sup_{S_{\Cu}}(\psi_i(\sup_{\overline S_i}(\gamma_i(s_i^j))_j))_i=
\sup_{S_{\Cu}}(\sup_{S_{\Cu}}(\psi_i(\gamma_i(s_i^j)))_j)_i=\sup_{S_{\Cu}}(\gamma(\varphi_i(s_i^j)))_{i,j},
\end{multline*}
and we see that each element in $S_{\Cu}$ can be written as the supremum of elements in  $\gamma(S_{\mathcal C})$.

Now let us prove that $\gamma$ is an order embedding.

Suppose $\gamma(s)\leq \gamma(t)$ in $S_{\mathcal C}$ and let $s=[(s_i)], t=[(t_i)]$ with rapidly increasing representatives as
in the construction of $S_{\mathcal C}$.
Then, recall that in this situation $s=\sup_{S_{\mathcal C}}(\varphi_i(s_i))_i$ and $t=\sup_{S_{\mathcal C}}(\varphi_i(t_i))_i$. Since $\gamma$ is
a $\mathcal C$ map and $\gamma\varphi_i=\psi_i\gamma_i$, we obtain
\[\sup_{S_{\Cu}}(\psi_i\gamma_i(s_i))_i\leq \sup_{S_{\Cu}}(\psi_i\gamma_i(t_i))_i.\]
Similarly as before, in $S_{\Cu}$ as constructed in \cite{CEI}, this is $[(\gamma_i(s_i))]\leq [(\gamma_i(t_i))]$.

Hence, given $x\ll s_i\in S_i$ we have $\gamma_i(x)\ll \gamma_i(s_i)$. By the order relation in $S_{\Cu}$ there exists $m\geq i$ such that 
$\overline{f_{m,i}}\gamma_i(x)\ll \gamma_m(t_m)$ and therefore $\gamma_m(f_{m,i}(x))\ll \gamma_m(t_m)$. But by definition $\gamma_m$ is an order embedding in PreCu,
which implies $f_{m,i}(x)\ll t_m$. But this, by the order relation in $S_{\mathcal C}$ implies $[(s_i)]\leq [(t_i)]$ or $s\leq t$. Hence
$\gamma$ is an order embedding, $(S_{\Cu},\gamma)$ is a completion of $S_{\mathcal C}$,  and by Theorem~\ref{universality} we have  $S_{\Cu}\cong \overline S_{\mathcal C}$.
\end{proof}

As a consequence, we can now compute the stabilized Cuntz semigroup for some countable
inductive limits of $C^*$-algebras in terms of the Classical Cuntz semigroup.

\begin{corollary} Let $A$ be a C$^*$-algebra such that $A=\lim_{\rightarrow}(A_i,f_i)$ where $\W(A_i)$ are hereditary. Then 
\[   \Cu (A)=\overline{\lim_{\rightarrow\mathcal C}(\W(A_i),\W(f_i))}.\]
\end{corollary}

\begin{proof}

Using \cite[Theorem, 2]{CEI}, Theorem~\ref{universality} and Theorem~\ref{thm:limitsac} we obtain

\[ \Cu (A)=\lim_{\rightarrow\Cu}(\Cu(A_i),\Cu(f_i))=\lim_{\rightarrow\Cu}\left(\overline{\W(A_i)},\overline{\W(f_i)}\right)=
\overline{\lim_{\rightarrow\mathcal C}(\W(A_i),\W(f_i))}.\]

\end{proof}


\end{document}